\documentclass[preprint,10pt]{elsarticle}

\usepackage{fullpage}
\usepackage{amsmath,amssymb,amsthm,graphicx,float,epsf}

\usepackage{fullpage}

\def\RR{\hbox{I\kern-.2em\hbox{R}}}

\newtheorem{theorem}{\indent Theorem}[section]

\newtheorem{lemma}{\indent Lemma}[section]

\newtheorem{corollary}[theorem]{\indent Corollary}

\newtheorem{example}[theorem]{\indent Example}

\title{Exponential stability for a system of second and first order delay differential equations
}

\author{Leonid Berezansky \\
Ben-Gurion University of the Negev, Dept. of Math., 
Beer-Sheva 84105, Israel \\
Elena Braverman
\\
 Dept. of Math \& Stats, University of Calgary, 
	Calgary, AB, Canada T2N1N4 }

\date{}
\begin{document}

\begin{abstract}
Exponential stability
of the second order linear delay differential equation in $x$ and $u$-control
$$
\ddot{x}(t)+a_1(t)\dot{x}(h_1(t))+a_2(t)x(h_2(t))+a_3(t)u(h_3(t))=0
$$
is studied, where indirect feedback control $\dot{u}(t)+b_1(t)u(g_1(t))+b_2(t)x(g_2(t))=0$ connects $u$ with the solution.
Explicit sufficient conditions guarantee that both $x$ and $u$ decay exponentially.



\end{abstract}

\maketitle

\section{Introduction and Preliminaries}

Ordinary differential equations (ODE) of the second order are one of the most interesting in applications and well investigated classes of differential equations. Delay differential equations (DDE) of the second order have also received a lot of attention, see \cite{ABBD,BDK,Burton2,Kolman,KN} and references therein.
To study stability, several methods were applied: checking that the roots of quasi-polynomials have negative real parts and Lyapunov-Krasovskii functionals~\cite{Kolman}, the fixed-point method~\cite{Burton2}, as well as the Bohl-Perron theorem and solution estimates~\cite{BDK,AS,B}.
A detailed study of stability properties for different classes of DDE
of the second order can be found  in the recent monograph \cite{BDK}. 
In particular, \cite{BDK} considers control problems for stabilization of DDE of the second order.

Some  real-world  models can be described by systems of two functional differential equations, where at least one of the equations is of the second order.
A simple model of ship stabilization \cite[Chapter 1, P. 4]{KN} is a system of two delay differential equations 
\begin{equation}
\label{rudder_syst}
\begin{array}{l}
I\ddot{\varphi}(t)+h\dot{\varphi}(t)=-K\psi(t)\\
T\dot{\psi}(t)+\psi(t)=\alpha\varphi(t-\tau)+\beta\dot{\varphi}(t-\tau),
\end{array}
\end{equation}
where $\varphi$ is the ship deviation angle and $\psi$ is the turning angle of the rudder, and
all the constants  are positive numbers. 

A harmonic oscillator $\ddot{x}(t)+b x(t)=0$ with an external force, a damping term, and the delay involved in both control and non-control terms, becomes
$$
\ddot{x}(t)+a\dot{x}(t-h_1)+bx(t-h_2)+cu(t-h_3)=f(t).
$$
Many motion equations with delay can be found in \cite{Kyrychko}. Further, motion control models can also be described by a system of DDE.
In control theory, two types of feedback systems are distinguished: direct control and indirect control systems. In many cases, explicitly defined signal $u$ cannot be used for control of a dynamical system, as it may be too weak. In this case, indirect feedback control is applied where an additional equation connecting $x$ and $u$ is introduced. 

For example, a simple model of movement control can be specified in the form of 
the following system
\begin{equation}
\begin{array}{l}
\ddot{x}(t)+a\dot{x}(t-h_1)+bx(t-h_2)+cu(t-h_3)=f(t),\\
\dot{u}(t)+du(t-g)+ex(t)=0,
\end{array}
\end{equation}
where the first equation is the  equation of the motion, and the second one describes indirect feedback control.
The function $u(t)$ is the control, $f(t)$ is  an external force, such as the power of a rocket engine.
For feasibility of such a model, any bounded right-hand side should lead to a bounded solution.
This requirement is satisfied if the system is exponentially stable.

Stability for a system of two equations, the first one of which is a second-order DDE, and second one is a first order DDE is the main object of this paper. 
We apply a priori estimates for $x$ and $u$ and their derivatives, Bohl-Perron theorem, $M$-matrices and matrix inequalities. For a scalar DDE  of the second order, these approaches are described in detail in \cite{BDK}.
Let us note that systems of several DDE of {\em different orders} have not been investigated before.

We consider the system of two DDE
\begin{equation}\label{2.1}
\ddot{x}(t)+a_1(t)\dot{x}(h_1(t))+a_2(t)x(h_2(t))+a_3(t)u(h_3(t))=0, ~~t\geq 0,
\end{equation}
\begin{equation}\label{2.2}
\dot{u}(t)+b_1(t)u(g_1(t))+b_2(t)x(g_2(t))=0,~~ t\geq 0,
\end{equation}
assuming without further mentioning that
\\
(a1) $a_i, b_i: [0,\infty) \to {\mathbb R}$ are Lebesgue measurable and  
essentially bounded, $0\leq t-h_i(t)\leq \tau_i$, $0\leq t-g_i(t)\leq \sigma_i$, where 
$h_i:[-\tau_i,\infty) \to {\mathbb R}$, $i=1,2,3$, $g_i(t):[-\sigma_i,\infty) \to {\mathbb R}$, $i=1,2$, are Lebesgue measurable,  $t \geq 0$.  

Non-homogeneous equivalent of \eqref{2.1}-\eqref{2.2} starting at $t_0\geq 0$ is 
\begin{equation}\label{2.3}
\ddot{x}(t)+a_1(t)\dot{x}(h_1(t))+a_2(t)x(h_2(t))+a_3(t)u(h_3(t))=f_1(t), ~~t\geq t_0,
\end{equation}
\begin{equation}\label{2.4}
\dot{u}(t)+b_1(t)u(g_1(t))+b_2(t)x(g_2(t))=f_2(t), ~~t\geq t_0.
\end{equation}
The initial conditions 
\begin{equation}\label{2.5}
x(t)=\varphi_1(t),\dot{x}(t)=\varphi_2(t), u(t)=\psi(t),~t\leq t_0
\end{equation} 
are everywhere assumed to satisfy
\\
(a2) $f_i:[t_0,\infty)\rightarrow {\mathbb R}$ are Lebesgue measurable 
essentially bounded functions, 
$\varphi_i:[t_0-\max_j \{\tau_j,\sigma_j \},t_0]\rightarrow {\mathbb R}$, 
$\psi:[t_0-\max_j \{\tau_j,\sigma_j \},t_0]\rightarrow {\mathbb R}$
are Borel measurable bounded functions.

The set of functions $x,u:{\mathbb R} \rightarrow {\mathbb R}$ with a locally 
absolutely continuous on $[t_0,\infty)$ derivatives $\dot{x}$ and $\dot{u}$ is 
{\bf a solution} of  problem \eqref{2.3}-\eqref{2.5} 
if \eqref{2.3}, \eqref{2.4} hold almost everywhere  for $t>t_0$ and \eqref{2.5} holds for $t\leq t_0$.

Some general considerations in \cite{AS} allow to state that the solution of  \ref{2.3}-\eqref{2.5} exists and is unique.

System \eqref{2.1}-\eqref{2.2} is {\bf uniformly exponentially 
stable}, if there exist
$M>0$, $\mu>0$,  such that  the solution of \eqref{2.3}-\eqref{2.5} with $f_1=f_2 \equiv 0$ satisfies
$$
\max\{|x(t)|,|u(t)|\}\leq M~e^{-\mu (t-t_0)}\left[ \sup_{t<t_0} \left( |\varphi_1(t)|+|\varphi_2(t)|+|\psi(t)|\right) \right],~t\geq t_0,
$$
where $M$ and $\mu$ do not depend on $t_ 0$ and on initial functions in \eqref{2.5}.
 
Further, ${\bf L}_{\infty}[t_0,\infty)$ is the space of all measurable essentially
bounded functions $y:[t_0,\infty) \to {\mathbb R}$ with the 
norm 
$
\|y\|_{[t_0,\infty)}=\mbox{ess}\sup_{t\geq t_0} |y(t)|,
$ 
similarly for any interval $\Omega=[t_0,t_1]$, $t_1>t_0$, 
${\bf L}_{\infty}(\Omega)$ has the norm
$
\|y\|_{\Omega}=\mbox{ess}\sup_{t\in \Omega} |y(t)|
$, 
${\bf C}[t_0,\infty)$ is the space of all continuous  bounded functions on  
$[t_0,\infty)$ with the $\sup$-norm. 


\begin{lemma}\cite[Bohl-Perron theorem]{AS}
\label{lemma2.2}
Suppose there exists $t_0\geq 0$ such that for every 
$f_i\in {\bf L}_{\infty}[t_0,\infty), i=1,2$ both functions of
 the solution $(x,u)$ of  problem \eqref{2.3}-\eqref{2.4}, where
\begin{equation}\label{2.6}
x(t)=0, \dot{x}(t)=0, u(t)=0,~t\leq t_0,
\end{equation}
belongs to ${\bf C}[t_0,\infty)$. Then system \eqref{2.1}-\eqref{2.2} is uniformly exponentially stable.
\end{lemma}

Consider an ordinary differential equation of the second order
\begin{equation}\label{2.7}
\ddot{x}(t)+a(t)\dot{x}(t)+b(t)x(t)=0.
\end{equation}
The fundamental function $X(t,s)$ of equation \eqref{2.7} is the solution of equation 
\eqref{2.7} for $t\geq s$ and initial conditions $x(s)=0, \dot{x}(s)=1$.
The solution of the initial value problem
$$
\ddot{x}(t)+a(t)\dot{x}(t)+b(t)x(t)=f(t), ~~t\geq t_0,
~~
x(t_0)=\dot{x}(t_0)=0
$$
has the form
$\displaystyle
x(t)=\int_{t_0}^t X(t,s) f(s)ds.
$

\begin{lemma}\cite[Page 551]{BDK}\label{lemma2.3}
Assume that $0<\alpha \leq a(t)\leq A$, $0< \beta \leq b(t)\leq B$, $\alpha^2\geq 4B$. Then the fundamental function of equation \eqref{2.7} is positive, $X(t,s) > 0$, $t \geq  s \geq t_0$ , \eqref{2.7}  is uniformly exponentially stable and
\begin{equation}\label{2.8}
\int_{t_0}^t X(t,s) b(s)ds\leq 1.
\end{equation}
\end{lemma}

We recall \cite{Berman} that a matrix $B=(b_{ij})_{i,j=1}^m$ is a non-singular $M$-matrix if $b_{ij}\leq 0$, $i\neq j$, 
and either $B$ is invertible with $B^{-1} \geq 0$,
or, equivalently, the leading principal minors of $B$ are positive.
If $B=I-A$, where $I$ is the identity matrix, 
and $A\geq 0$, $B$ is an $M$-matrix if and only if the spectral radius of $A$ is less than one.

\section{The Main Result}

Introduce a $5\times 5$ matrix
\begin{equation}\label{3.0}
A=\left(\begin{array}{ccccc}
0&\tau_2&\tau_1\left\|\frac{a_1}{a_2}\right\|_{[t_0,\infty)}& \left\|\frac{a_3}{a_2}\right\|_{[t_0,\infty)}&0\\
\left\|\frac{a_2}{a_1}\right\|_{[t_0,\infty)}&0&\tau_1&\left\|\frac{a_3}{a_1}\right\|_{[t_0,\infty)}&0\\
\|a_2\|_{[t_0,\infty)}& \|a_1\|_{[t_0,\infty)}&0&\|a_3\|_{[t_0,\infty)}&0\\
\left\|\frac{b_2}{b_1}\right\|_{[t_0,\infty)}&0&0&0&\sigma_1\\
\|b_2\|_{[t_0,\infty)}&0&0&\|b_1\|_{[t_0,\infty)}&0\\
\end{array}\right).
\end{equation}

\begin{theorem}\label{theorem3.1}
Assume that for some $t_0\geq 0$, the inequalities
$
0<\alpha_i\leq a_i(t)\leq A_i$, $i=1,2$, $|a_3(t)|\leq A_3$,
$0 < \beta_1 \leq b_1(t) \leq  B_1$,  $|b_2(t)|\leq B_2$, $\alpha_1^2 \geq 4A_2$
hold for $t \geq t_0$, and either the spectral radius of $A$ in \eqref{3.0} is less than one or, which is equivalent, $B=I-A$ is an $M$-matrix.

Then system \eqref{2.1},\eqref{2.2} is uniformly exponentially stable.
\end{theorem}

\begin{proof}
To apply Lemma \ref{lemma2.2}, consider system \eqref{2.3}-\eqref{2.4} with essentially bounded on $[t_0,\infty)$
functions $f_i, i=1,2$ and initial conditions \eqref{2.6}. Let $t_1> t_0$, $\Omega=[t_0,t_1]$.

First, we estimate $\| x \|_{\Omega}$ and its derivatives, then  $\| u \|_{\Omega}$ and its derivative.

1) From equation \eqref{2.3} we have
\begin{equation}
\label{3.1}
\|\ddot{x}\|_{\Omega} \leq \|a_1\|_{[t_0,\infty)}\|\dot{x}\|_{\Omega}+\|a_2\|_{[t_0,\infty)}\|x\|_{\Omega}+
\|a_3\|_{[t_0,\infty)}\|u\|_{\Omega}+
\|f_1\|_{[t_0,\infty)}.
\end{equation}

Let us rewrite equation \eqref{2.3} as
$$
\ddot{x}(t)+a_1(t)\dot{x}(t)=a_1(t)\int_{h_1(t)}^t \ddot{x}(\xi)d\xi-a_2(t)x(h_2(t))-a_3(t)u(h_3(t))+f_1(t),
$$
leading to
$$
\dot{x}(t)=\int_{t_0}^t e^{-\int_s^t a_1(\xi)d\xi}a_1(s)\left[\int_{h_1(s)}^s\ddot{x}(\xi)d\xi-\frac{a_2(s)}{a_1(s)}x(h_2(s))
-\frac{a_3(s)}{a_1(s)}u(h_3(s))\right]ds+f_3(t),
$$
where $f_3(t)=\int_{t_0}^t e^{-\int_s^t a_1(\xi)d\xi}f_1(s)ds\in {\bf L}_{\infty}[t_0,\infty)$.
Since $\displaystyle \left| \int_{t_0}^t e^{-\int _s^t c(\xi)d\xi}c(s) g(s)~ds \right| \leq 
\| g \|_{[t_0,\infty)}$, $t > t_0$ for $c(t)\geq c_0>0$, we have
\begin{equation}\label{3.2}
\|\dot{x}\|_{\Omega} \leq \tau_1\|\ddot{x}\|_{\Omega}+\left\|\frac{a_2}{a_1}\right\|_{[t_0,\infty)}\|x\|_{\Omega}
+\left\|\frac{a_3}{a_1}\right\|_{[t_0,\infty)}\|u\|_{\Omega} +\|f_3\|_{[t_0,\infty)}.
\end{equation}
Further, \eqref{2.3}  can be written in a different form
$$
\ddot{x}(t)+a_1(t)\dot{x}(t)+a_2(t)x(t)=a_1(t)\int\limits_{h_1(t)}^t\ddot{x}(\xi)d\xi
+a_2(t)\int\limits_{h_2(t)}^t\dot{x}(\xi)d\xi-a_3(t)u(h_3(t))+f_1(t).
$$
Denote by $X(t,s)$ the fundamental function of the equation 
$
\ddot{x}(t)+a_1(t)\dot{x}(t)+a_2(t)x(t)=0.
$
By Lemma \ref{lemma2.3}, $X(t,s)>0$. Also,
$$
x(t)=\int_{t_0}^t X(t,s) a_2(s)\left[\frac{a_1(s)}{a_2(s)}\int_{h_1(s)}^s \ddot{x}(\xi)d\xi   
+\int_{h_2(s)}^s \dot{x}(\xi)d\xi-\frac{a_3(s)}{a_2(s)}u(h_3(s))\right]+f_4(t),
$$
where $f_4(t)=\int_{t_0}^t X(t,s)f_1(s)ds\in {\bf L}_{\infty}[t_0,\infty)$.
Lemma~\ref{lemma2.3} implies
\begin{equation}\label{3.3}
\|x\|_{\Omega} \leq \tau_1 \left\|\frac{a_1}{a_2}\right\|_{[t_0,\infty)}\|\ddot{x}\|_{\Omega}+\tau_2\|\dot{x}\|_{\Omega}
+\left\|\frac{a_3}{a_2}\right\|_{[t_0,\infty)}\|u\|_{\Omega}+\|f_4\|_{[t_0,\infty)}.
\end{equation}

2) Next, let us estimate $u$ and $\dot{u}$.
From equation \eqref{2.4},
\begin{equation}\label{3.4}
\|\dot{u}\|_{\Omega} \leq \|b_1\|_{[t_0,\infty)}\|u\|_{\Omega} +\|b_2\|_{[t_0,\infty)}\|x\|_{\Omega} +\|f_2\|_{[t_0,\infty)}.
\end{equation}
Rewriting equation \eqref{2.4} as
$$
\dot{u}(t)+b_1(t)u(t)=b_1(t)\int_{g_1(t)}^t \dot{u}(\xi)d\xi-b_2(t)x(g_2(t)) +f_2(t),
$$
we get for $f_5(t)=\int_{t_0}^t  e^{-\int_s^t b_1(\xi)d\xi}f_2(s)ds\in {\bf L}_{\infty}[t_0,\infty)$,
$$
u(t)=\int_{t_0}^t e^{-\int_s^t b_1(\xi)d\xi}b_1(s)\left[\int_{g_1(s)}^s \dot{u}(\xi)d\xi-\frac{b_2(s)}{b_1(s)}x(g_2(s))\right]ds
+f_5(t).
$$

Hence
\begin{equation}\label{3.5}
\|u\|_{\Omega} \leq \left\|\frac{b_2}{b_1}\right\|_{[t_0,\infty)}\|x\|_{\Omega}+\sigma_1\|\dot{u}\|_{\Omega} +\|f_5\|_{[t_0,\infty)}.
\end{equation}

Denote two five-dimensional column vectors
$X_u=(\|x\|_{\Omega},\|\dot{x}\|_{\Omega}, \|\ddot{x}\|_{\Omega},\|u\|_{\Omega},\|\dot{u}\|_{\Omega})$ and \\
$F=(\|f_4\|_{[t_0,\infty)},\|f_3\|_{[t_0,\infty)},
\|f_1\|_{[t_0,\infty)},\|f_5\|_{[t_0,\infty)},\|f_2\|_{[t_0,\infty)})$.
Inequalities \eqref{3.3},\eqref{3.2},\eqref{3.1},\eqref{3.5},\eqref{3.4}
can be rewritten in the matrix form 
$
X_u\leq A X_u+F$, 
where the matrix $A$ was introduced in \eqref{3.0}.
Since the spectral radius of the non-negative matrix $A$ is less than one, there exists a non-negative inverse matrix $(I-A)^{-1}$.
Hence $X_u\leq A X_u+F$ implies
\begin{equation}\label{3.12}
X_u\leq (I-A)^{-1}F, 
\end{equation}
where $I$ is the identity matrix.
The right hand side of \eqref{3.12} does not depend on $t_1$, therefore the two solutions $(x,u)$ of system \eqref{2.3}-\eqref{2.4} with the zero initial conditions are bounded functions on the interval $[t_0,\infty)$. By Lemma \ref{lemma2.2}, system \eqref{2.1}-\eqref{2.2} is uniformly exponentially stable.
\end{proof}

Next, consider ODE 
\begin{equation}\label{3.13}
\ddot{x}(t)+a_1(t)\dot{x}(t)+a_2(t)x(t)+a_3(t)u(h(t))=0
\end{equation}
with delayed indirect feedback control \eqref{2.2}. Assume that for \eqref{3.13} condition (a1) holds, in particular, that
$0 \leq t-h(t)\leq \tau$, $0 \leq t-g_i(t)\leq \sigma_i$, $i=1,2$. 

\begin{corollary}\label{c3.1}
Assume for some $t_0\geq 0$,
$0<\alpha_i\leq a_i(t)\leq A_i$, $i=1,2$, $|a_3(t)|\leq A_3$,
$0<\beta_1\leq b_1(t)\leq B_1$, $|b_2(t)|\leq B_2$, $\alpha_1^2\geq 4A_2$,
and 
\begin{equation}
\label{3.14}
\sigma_1\|b_1\|_{[t_0,\infty)}+ \sigma_1 \| b_2\|_{[t_0, \infty)} \left\| \frac{a_3}{a_2} \right\|_{[t_0, \infty)} +
\left\|\frac{a_3}{a_2}\right\|_{[t_0,\infty)}\left\|\frac{b_2}{b_1}\right\|_{[t_0,\infty)}
<1.
\end{equation}
Then system \eqref{3.13},\eqref{2.2} is uniformly exponentially stable.
\end{corollary}
\begin{proof}
The matrix $A$ denoted by \eqref{3.0} for system \eqref{3.13},\eqref{2.2} is
$$
\left(\begin{array}{lllll}
0&0&0& \left\|\frac{a_3}{a_2}\right\|_{[t_0,\infty)}&0\\
\left\|\frac{a_2}{a_1}\right\|_{[t_0,\infty)}&0&0&\left\|\frac{a_3}{a_1}\right\|_{[t_0,\infty)}&0\\
\|a_2\|_{[t_0,\infty)}& \|a_1\|_{[t_0,\infty)}&0&\|a_3\|_{[t_0,\infty)}&0\\
\left\|\frac{b_2}{b_1}\right\|_{[t_0,\infty)}&0&0&0&\sigma_1\\
\|b_2\|_{[t_0,\infty)}&0&0&\|b_1\|_{[t_0,\infty)}&0\\
\end{array}\right),
$$
and $B=I-A$ has the form
$$
\left(\begin{array}{ccccc}
1&0&0& -\left\|\frac{a_3}{a_2}\right\|_{[t_0,\infty)}&0\\
-\left\|\frac{a_2}{a_1}\right\|_{[t_0,\infty)}&1&0&-\left\|\frac{a_3}{a_1}\right\|_{[t_0,\infty)}&0\\
-\|a_2\|_{[t_0,\infty)}& -\|a_1\|_{[t_0,\infty)}&1&-\|a_3\|_{[t_0,\infty)}&0\\
-\left\|\frac{b_2}{b_1}\right\|_{[t_0,\infty)}&0&0&1&-\sigma_1\\
-\|b_2\|_{[t_0,\infty)}&0&0&-\|b_1\|_{[t_0,\infty)}&1\\
\end{array}\right).
$$
Inequality \eqref{3.14} implies that all the leading principal minors of the second matrix $B$ 
are positive, thus $B$ is an $M$-matrix. By Theorem~\ref{theorem3.1},  system \eqref{3.13},\eqref{2.2} is uniformly exponentially stable.
\end{proof}

\section{Example and Discussion}

\begin{example}
Consider the system
\begin{equation}\label{4.1}
\begin{array}{ll}
& \ddot{x}(t)+(1+0.01 |\sin (t)|)\dot{x}(t-0.1|\sin (3t)|) \\
+ & (0.2+0.05|\cos (t)|)x(t-0.1|\cos (3t)|)
-0.1\sin (10 t) u(t-8 \sin^2 (5t))=0, 
\end{array}
\end{equation}
\begin{equation}\label{4.2}
\dot{u}(t)+ \left( 0.2+0.1|\cos (2t)| \right) u(t-0.1 \sin^2 (t))+0.1 \cos (t) x(t-5 \cos^2 (3t))=0.
\end{equation}
To apply Theorem~\ref{theorem3.1},  denote
$
a_1(t)=1+0.01 |\sin (t)|, a_2(t)=0.2+0.05|\cos (t)|$, $a_3(t)=-0.1\sin (10 t), 
$
$
b_1(t)=0.2+0.1|\cos (2t)|$, $b_2(t)=0.1 \cos (t)$,
$1\leq a_1(t)\leq 1.01$, $0.2\leq a_2(t)\leq 0.25$, $|a_3(t)|\leq 0.1$,
$
0.2\leq b_1(t)\leq 0.3$, $|b_2(t)|\leq 0.1$, $\tau_1=\tau_2=0.1$, $\tau_3=8$, $\sigma_1=0.1$, $\sigma_2=5$.
Since $\alpha_1=1, A_2=0.25$, the condition  $(\alpha_1)^2\geq 4A_2$ holds.
We have 
$$
\left\|\frac{a_1}{a_2}\right\|_{[t_0,\infty)}\leq 5.05,~~ \left\|\frac{a_3}{a_2}\right\|_{[t_0,\infty)}\leq 0.5, ~~
\left\|\frac{a_2}{a_1}\right\|_{[t_0,\infty)}\leq 0.25, 
~~ \left\|\frac{a_3}{a_1}\right\|_{[t_0,\infty)}\leq 0.1,
$$$$
\|a_1\|=1.01,~~\|a_2\|=0.25, ~~\|a_3\|=0.1, \left\|\frac{b_2}{b_1}\right\|_{[t_0,\infty)}\leq 0.5,
\|b_1\|=0.3, \|b_2\|=0.1.
$$
The matrix $A$ for system \eqref{4.1},\eqref{4.2} satisfies $0\leq A\leq \tilde{A}$, where
$
\displaystyle \tilde{A}=\left(\begin{array}{lllll}
0&0.1&0.505& 0.5&0\\
0.25&0&0.1&0.1&0\\
0.25&1.01&0&0.1&0\\
0.5&0&0&0&0.1\\
0.1&0&0&0.3&0\\
\end{array}\right).
$

Since $0\leq A\leq \tilde{A}$ then \cite[P.27, Corollary 1.5 (a)]{Berman}  for the spectral radius
$r(A) \leq r(\tilde{A})$.
We evaluate $r(\tilde{A}) \approx 0.8443<1$ 
numerically (using MATLAB or Wolfram alpha). 
By Theorem \ref{theorem3.1}, system \eqref{4.1},\eqref{4.2}
is uniformly exponentially stable.
\end{example}

Note that if $a_3(t)\equiv 0, b_2(t)\equiv0$ then system \eqref{2.1}-\eqref{2.2} transforms into two independent equations, the first one of the second order and the second of the first order. In this case Theorem~\eqref{theorem3.1}  implies 
\\
1) Equation 
$
\ddot{x}(t)+a_1(t)\dot{x}(h_1(t))+a_2(t)x(h_2(t))=0
$
is uniformly exponentially stable if the inequalities
$ 0<\alpha_i\leq a_i(t)\leq A_i$, $t-h_i(t)\leq \tau_i$, $i=1,2$, $\alpha_1^2\geq 4A_2$
hold for $t \geq t_0$, and the spectral radius of the matrix 
$\displaystyle \left(\begin{array}{ccccc}
0&\tau_2&\tau_1\left\|\frac{a_1}{a_2}\right\|_{[t_0,\infty)}& 0&0\\
\left\|\frac{a_2}{a_1}\right\|_{[t_0,\infty)}&0&0&0\\
\|a_2\|_{[t_0,\infty)}& \|a_1\|_{[t_0,\infty)}&0&0&0\\
0&0&0&0&\sigma_1\\
0&0&0&\|b_1\|_{[t_0,\infty)}&0\\
\end{array}\right)$ is less than one.


This stability condition coincides with a corollary of \cite[Theorem 1]{B}.

2) Equation $\dot{u}(t)+b_1(t)u(g_1(t))=0$ is uniformly exponentially stable if $0<\beta\leq b_1(t)\leq B_1$, $t-g_1(t)\leq \sigma_1$, $t\geq t_0$ and
$\sigma_1 B_1<1$. This stability condition is well-known, moreover, the constant 1 can be replaced with $\frac{3}{2}$
\cite{SB}.

\section*{Acknowledgment}

The second author acknowledges the support of NSERC, the grant RGPIN-2020-03934.
The authors are very grateful to the anonymous referees whose thoughtful comments significantly contributed to the paper presentation.

\end{document}